\documentclass[]{amsart}

\usepackage{latexsym}
\usepackage{amssymb,amsmath,amsopn}
\usepackage[dvips]{graphicx}   
\usepackage{color,epsfig}      
\usepackage{url}

\newtheorem{theorem}{Theorem}[]
\newtheorem{lemma}{Lemma}[]
\newtheorem{statement}{Statement}[]

\theoremstyle{definition}
\newtheorem{defi}{Definition}[]

\newtheorem{remark}[]{Remark}[]
\newtheorem{example}[]{Example}[]

\DeclareMathOperator{\vol}{vol}

\newcommand{\R}{\mathrm{R}}
\newcommand{\Q}{\mathrm{Q}}
\newcommand{\E}{\mathrm{E}}

\title{Angle measures, general rotations, and roulettes in normed planes }

\author[Vitor Balestro]{Vitor Balestro}
\address {CEFET/RJ Campus Nova Friburgo\\
Nova Friburgo\\
BRAZIL}
\email{vitorbalestro@mat.uff.br}
\author[\'A. G.Horv\'ath]{\'Akos G.Horv\'ath}
\address {Department of Geometry \\
Budapest University of Technology and Economics\\
H-1521 Budapest\\
Hungary}
\email{ghorvath@math.bme.hu}
\author[H. Martini]{Horst Martini}
\address {Faculty of Mathematics \\ Chemnitz University of Technology\\ 09107 Chemnitz, Germany}
\email{horst.martini@mathematik.tu-chemnitz.de}
\date{Jan, 2016}

\begin{document}

\begin{abstract}
In this paper  a special group of bijective maps of a normed plane, called the group of general rotations, is introduced; it contains the isometry group as a subgroup. The concept of general rotations leads to the notion of flexible motions of the plane, and to the concept of Minkowskian roulettes. As a nice consequence of this new approach to motions the validity of a strong analogue to the Euler-Savary equations for Minkowskian roulettes is proved.
\end{abstract}

\maketitle

\emph{Keywords:} angle measure; Busemann curvature; Euler-Savary equations; Finsler space; normed plane; roulettes

\emph{2010 MSC:} 46B20, 51M05, 52A21, 53A17

\section{Introduction}

    In the spirit of Minkowski geometry (see \cite{thompson}), our considerations refer to \emph{normed} (or \emph{Minkowski}) \emph{planes}, i.e., to two-dimensional real Banach spaces, whose \emph{unit ball} $B$ is a two-dimensional compact, convex set centered at the origin $O$ of a cartesian coordinate system (which itself presents the Euclidean background metric). The \emph{unit circle} of a normed plane is the boundary $\partial B$ of $B$.
We write capital letters like $A,B,\dots$ for \emph{points} with respective \emph{position vectors} ${\bf a},{\bf b}, \dots$; by $a,b,\dots, g(A,B)$ we denote \emph{lines}, in the latter case spanned by $A$ and $B$, and by $AB$ the \emph{segment} with endpoints $A$ and $B$ is meant. We use $\overrightarrow{AB}$ for the \emph{vector} from $A$ to $B$, or for the \emph{half-line} starting at $A$ and passing through $B$; sometimes we use also $a,b,\dots r_1,r_2$ for half-lines (the respective meaning will be clear by the context). Further on, we write $\|{\bf a}\|, \|{\bf a}\|_E$ for the \emph{general Minkowskian} and the \emph{Euclidean norm} of ${\bf a}$, respectively, and ${\bf a}^o$ stands for the \emph{Minkowskian unit vector} parallel to ${\bf a}; [{\bf a},{\bf b}]$ is the \emph{semi-inner product} corresponding to the Minkowskian norm $\| \cdot \|$. For \emph{identity} and \emph{interior} we use id and int, respectively. Referring to the Minkowskian arc-length $s$, we denote by $r(s)$ the \emph{radial function} of the Minkowskian unit circle, and by $\gamma(s)$ a \emph{planar curve}, both parametrized by $s$; $\chi_\gamma(s)$ is the \emph{Busemann curvature function} of $\gamma(s)$. The \emph{Busemann sigma function} of the $r$-dimensional affine subspace $V_r$ is $\sigma(V_r)$, and $(a,b) \angle$ denotes the \emph{angle} determined by the lines $a,b$. Finally, we say that a vector $x$ is \emph{Birkhoff orthogonal} to a vector $y$ if $||x+ty|| \geq ||x||$ for every $t \in \mathbb{R}$.

\section{Angle measure and rotations in Minkowski planes}

The first part of the following subsection contains also some history of the subject studied here.

\subsection{Angle measures}

The question how to measure angles is old and interesting. Inspired by  Hilbert's axiomatic approach to geometry (see \cite{Hil}), many authors delt with this problem in a large variety of interesting situations. E.g., for continuously differentiable curves satisfying a general extremal property this was discussed by Bliss \cite{bliss}. He defined his concept as follows: If $OA'$ and $OA$ are two extremal rays through the point $O$, and
$A'A$ is the arc of length $l$ of a transversal (which is an arc of a generalized circle passing through $A$ and $A'$, with center $O$) at the generalized distance $r$ from
$O$, then the \emph{generalized angle} between $OA'$ and $OA$ is defined to be the limit
of the ratio $l/r$ as $r$ approaches zero. His analytical formulas reflect the usual computation methods in classical Euclidean and non-Euclidean geometries, and even on surfaces embedded into the Euclidean 3-space.

Busemann \cite{busemann 1} investigated the geometry of Finsler spaces, and there he observed the following facts:
\emph{The volume problem makes it more than probable that an analogous
situation\footnote{as the Riemannian one} exists for Finsler spaces. Therefore the study of Minkowskian
geometry ought to be the first and main step, the passage from there to
general Finsler spaces will be the second and simpler step.
What has been done in Minkowskian geometry, what are the difficulties
and problems, and which tools will be necessary? Little has
been done, but the field is quite accessible. The main difficulty comes
from our long Euclidean tradition, which makes it hard (at least for
the author) to get a feeling for the subject and to conjecture the
right theorems.
The type of problem which faces us is clear: A Minkowskian geometry
admits in general only the translations as motions and not the
rotations. Since the group of motions is smaller, we expect more invariants.
By passing from Euclidean to projective geometry, ellipses,
parabolas, and hyperbolas become indiscernible. The present case
presents \texttt{ the much more difficult converse problem, to discern objects
which have always been considered as identical}.} To illustrate this, Busemann noted that, contrarily to the Euclidean case, the general Minkowskian sphere defined by the set of points with the given Minkowskian distance $r$ from the origin $O$ holds the first property, but does not hold the third and fourth property of the list given here:
\begin{itemize}
\item It maximizes the volume among all sets of diameter $2r$.
\item It is envelope of planes normal to the rays with origin $r$\footnote{Normality means the so-called Birkhoff orthogonality of the Minkowski plane}.
\item It solves the isoperimetric problem.
\item It leads to the area $A(S)$ of a convex or sufficiently smooth non-convex simple closed surface $S$ bounding a set $K$ via the relation $A(S) =\lim_{r\rightarrow 0}(\vol(K(r))-\vol (K))/r $, where $K(r)$ is the outer parallel domain of the body $K$ of radius $r$.
\end{itemize}

The solution of the third property is another set, the so-called \emph{isoperimetrix} with respect to the sphere $B$. In the planar situation it is obtained by rotating the Euclidean \emph{polar} of $B$ by $\pi$, and thus it is also a convex body centered at the origin. The isoperimetrix can also be defined in an intrinsic way, using the concept of \emph{antinorm} (see \cite{martini-swanepoel}). This definition can be extended in a natural way for every even dimensional space (see \cite{gho-langi-spirova}),
and this convex body solves the fourth property if we take its enlarged copies to determine the parallel domain. The second property can be divided into two properties; the sphere $B$ has one of them, and the solution of the isoperimetric problem has the other one.

In \cite{busemann}, Busemann discussed the "axiom" for angle measures in the case of plane curves belonging to a class $\mathcal{S}$ of open Jordan curves, holding the additional property that any two distinct points lie on exactly one curve of $\mathcal{S}$. He defined the concepts of \textit{ray} $r$, \textit{angle} $D$ with \emph{legs} $r_1$ and $r_2$, and \textit{angle measure} $|D|$ on the set of angles having the following properties:
\begin{enumerate}
\item $|D|\geq 0$ (positivity),
\item $|D|=\pi$ if and only if $D$ is straight,
\item if $D_1$ and $D_2$ are two angles with a common leg but with no other common ray, then $|D_1\cup D_2|=|D_1|+|D_2|$ (additivity),
\item if $D_\nu \rightarrow D$, then $|D_\nu|\rightarrow D$ (continuity).
\end{enumerate}
He showed that these assumptions are sufficient to obtain many of the usual relationships
between angle measure and curvature. We note that Busemann collected the essential properties of an angle measure  that we have to require in every structure, where a natural concept of angle exists.

Lippmann \cite{lippmann} considered the classical Minkowski space defined on the $n$-dimensional Euclidean space by a "metrische Grundfunction" $F$, which is a positive, convex functional on the space being
homogeneous of first degree. In our terminology, $F$ is the \emph{norm-square function}. To have convexity (following Minkowski's definition), Lippmann required continuity of the second partial derivative, and positivity of the second derivative of $F$. Hence the unit ball of the corresponding space is always smooth. He used the arcus cosine of the bivariate function
$$
(x,y):=\frac{\sum x_i\frac{\partial }{\partial x_i}F(y)}{F(x)}
$$
to measure the angle between $x$ and $y$. This yields a concept of transversality, namely: $x$ is \emph{transversal} to $y$ if $(x, y)=0$. A wide variety
of angle measures referring to metric properties can be found in the literature. E.g., Lippmann's papers \cite{lippmannI,lippmannII} contain typically metric definitions of angle measures. For the situation in (normed or) Minkowski planes see, in addition to the papers already mentioned, Graham, Witsenhausen and Zassenhaus \cite{graham}. This paper refers to a useful metrical classification of angles by their measures, and a good review on this topic can be found in the book of Thompson \cite{thompson}.

In the last few decades some authors rediscovered this interesting problem in connection with the problem of orthogonality. We have to mention P. Brass who in \cite{brass} redefined the concept of angle measure as follows.

\begin{defi}\label{brassmeasure}
By an \textit{angle measure} we mean a measure $\mu$ on the unit circle $\partial B$ with center $O$ which
is extended in the usual translation-invariant way to measure angles elsewhere, and which has the
following properties:
\begin{enumerate}
\item $\mu(\partial B) = 2\pi$,
\item for any Borel set $S\subset \partial B$ we have $\mu(S) = \mu(-S)$, and
\item for each $p \in \partial B$ we have $\mu(\{p\}) = 0$.
\end{enumerate}
\end{defi}

This concept was used in the papers of D\"uvelmeyer \cite{duvelmeyer}, Martini and Swanepoel \cite{martini-swanepoel}, and Fankh\"anel \cite{fankhanel 1,fankhanel 2}.

Another direction of research is to give immediate metric definitions of the angle of two vectors. In this direction we can find also papers of P. M. Mili\v{c}i\v{c} \cite{milicic}, C. R. Diminnie, E. Z. Andalafte, R. W. Freese \cite{diminnie} or H. Gunawan, J. Lindiarni and O. Neswan \cite{gunawan}. Further related papers on angle measures are \cite{Dek1}, \cite{Dek2}, \cite{Dek3}, and \cite{Ling}.

As Busemann observed, the problem to find a natural definition of angular measure arises from the fact that the group of Minkowski rotations is very small. In a general normed space there are no such rotations which are also isometries of the space. More precisely, if $(V, \|\cdot\|)$ is a \emph{Minkowski plane} that is non-Euclidean, then the group $\mathcal{I}(2)$ of isometries of $(V, \|\cdot\|)$ is isomorphic to the semi-direct product of the translation group $\mathcal{T}(2)$ of $\mathbb{R}^2$ with a finite group of even order that is either a cyclic group of rotations or a dihedral group (see \cite{garcia}, \cite{thompson}, \cite{Ma-Sp}, and \cite{martini-spirova-strambach}). On the other hand, there are so-called left reflections (right-reflections) based on the notion of Birkhoff orthogonality (see \cite{Ma-Sp} and \cite{martini-spirova-strambach}). These are not isometries, but they have some important properties of isometries; e.g., they are affine mappings of the plane sending lines into lines; the product of three left reflections in parallel lines in a strictly convex Minkowski
plane is a left reflection in another line belonging to the same pencil of parallel lines; and the product of two left reflections in Birkhoff orthogonal lines is a symmetry of the plane. Unfortunately, if in a strictly convex and smooth Minkowski plane for left reflections the main lemma on three reflections with concurrent axes holds, then the plane is already Euclidean. Hence there is no chance to define an angle measure and also rotations by left reflections in the way that "a rotation is the product of two left reflections in non-parallel lines". This motivates our definition of Minkowski rotations.

\subsection{General rotations}

In order to define a concept of rotation for a Minkowski plane, we start with extending the definition of Brass by considering Borel measures in a larger class of curves, not only in the unit circle, and we will derive  angle measures for normed planes from it.

\begin{defi}\label{genmeasure}\normalfont Let $\gamma \subseteq X$ be a closed Jordan curve which is starlike with respect to a point $p$ of the interior of the region bounded by $\gamma$. An \emph{angle measure with respect to such a Jordan curve} is a (normalized) Borel measure $\mu_{\gamma}$ on $\gamma$ for which the following properties hold: \\

\noindent\textbf{(a)} $\mu_{\gamma}(\gamma) = 2\pi$;

\noindent\textbf{(b)} for any $q \in \gamma$ we have $\mu_{\gamma}(\{q\}) = 0$; and

\noindent\textbf{(c)} any non-degenerate arc of $\gamma$ has positive measure.\\

 An angle measure defined in this way provides a translation invariant measure of \emph{angles} in the plane, which we define to be the convex hulls of two rays with the same starting point, or the half-plane given by two opposite rays. Given an angle $(r_1,r_2)\angle$ with apex $a$, we define its \textit{generalized angle measure} $\mu_{\gamma,p}(r_1,r_2)$ to be the measure $\mu_{\gamma}$ of the arc determined on $\gamma$ by the image of $(r_1,r_2)\angle$ via the translation $x \mapsto x - a + p$. Figure \ref{figgenang} illustrates this concept.
\end{defi}

\begin{figure}[h]
\centering
\includegraphics{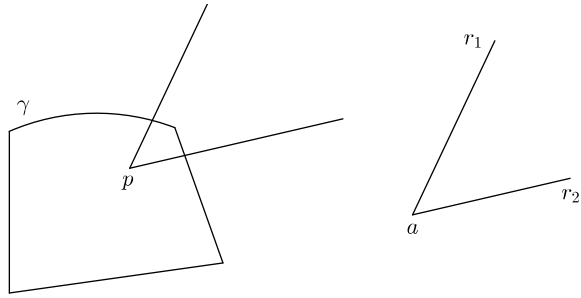}
\caption{The generalized angle measure given by $\mu_{\gamma}$ and $p$}
\label{figgenang}
\end{figure}

Using this notion of generalized angle measure we define now the generalized rotations in Minkowski planes. \\

\begin{defi}\label{generalrotation}\normalfont Let $(X,||\cdot||)$ be a Minkowski plane and let $\gamma$ be a closed Jordan curve which is starlike with respect to a point $p$ of the interior of the region bounded by $\gamma$. Let $\mu_{\gamma,p}$ be a generalized angle measure as in the previous definition. A \textit{general rotation} (with respect to $\mu_{\gamma,p}$) is a transform $\mathrm{rot}_{\mu_{\gamma,p}}:X \rightarrow X$ for which the following three properties hold: \\

\noindent\textbf{(a)} The transform $\mathrm{rot}_{\mu_{\gamma,p}}$ leaves invariant the pencil $\mathcal{R}(p)$ of rays with origin in $p$. In other words, if $r \subseteq X$ is a ray with origin $p$, then $\mathrm{rot}_{\mu_{\gamma,p}}(r)$ is also a ray with origin $p$.

\noindent\textbf{(b)} For each $\alpha > 0$, $\mathrm{rot}_{\mu_{\gamma,p}}$ leaves invariant the homothetic curve $\gamma_{\alpha,p} := p + \alpha(\gamma - p)$, i.e., for such a curve we have $\mathrm{rot}_{\mu_{\gamma,p}}\left(\gamma_{\alpha,p}\right) \subseteq \gamma_{\alpha,p}$.

\noindent\textbf{(c)} The function $r \in \mathcal{R}(p) \mapsto \mu_{\gamma,p}\left(\mathrm{rot}_{\mu_{\gamma,p}}(r),r\right)$ is constant. Intuitively, $\mathrm{rot}_{\mu_{\gamma,p}}$ ``rotates every ray of $\mathcal{R}(p)$ by a same angle".
\end{defi}

\begin{remark}\normalfont Notice that a general rotation can be considered as acting in the space of directions of $X$. Indeed, the set $\mathcal{R}(p)$ can be seen as this space. Later this viewpoint will be useful.
\end{remark}

We emphasize that any general rotation relies on a fixed closed Jordan curve $\gamma$, an inner point $p$ with respect to which $\gamma$ is starlike, and a generalized angle measure $\mu_{\gamma,p}$. On the other hand, these three informations yield a certain class of general rotations, which we denote by $\mathcal{R}(\gamma,\mu,p)$. We head now to describe an element of such a class in terms of the angle of rotation. For any $\theta \in [0,2\pi)$ we set $\mathrm{rot}_{\theta}:X\rightarrow X$ as follows: if $q_1 \in \gamma$, then $q_1$ is mapped to the (unique) point $q_2 \in \gamma$ taken counterclockwise, say, for which the rays $r_1=\left.[p,q_1\right>$ and $r_2=\left.[p,q_2\right>$ are such that $\mu(r_1,r_2) = \theta$. Now, any point $q \in X\setminus\gamma$ can be written in the form $q = p + \alpha\left(\mathrm{rad}_{\gamma,p}(\left.[p,q\right>)-p\right)$ for some $\alpha \geq 0$, where $\mathrm{rad}_{\gamma,p}:\mathcal{R}(p)\rightarrow \gamma$ is the \emph{radial function} which associates each ray starting at $p$ to its intersection with $\gamma$. We just set
\begin{align*} \mathrm{rot}_{\theta}(q) = p + \alpha\left(\mathrm{rot}_{\theta}\left(\mathrm{rad}_{\gamma,p}\left(\left.[p,q\right>\right)\right) - p\right). \end{align*}

It is clear that $\mathcal{R}(\gamma,\mu,p) = \{\mathrm{rot}_{\theta}\}_{\theta\in[0,2\pi)}$. This description indicates that a class $\mathcal{R}(\gamma,\mu,p)$ has a group structure under composition, as in the standard Euclidean case. This is summarized in the following lemma.

\begin{lemma} For a class $\mathcal{R}(\gamma,\mu,p)$ we have the following properties:

\normalfont
\noindent\textbf{(a)} \emph{Regarding composition}, \textit{$\mathcal{R}(\gamma,\mu,p)$ is an abelian group. More precisely, we have $\mathrm{rot}_{\theta_1}\circ\mathrm{rot}_{\theta_2} = \mathrm{rot}_{\theta_1\oplus \theta_2}$, where $\oplus$ is the sum modulo $2\pi$.}

\noindent\textbf{(b)} \textit{For any $q\in\gamma$, the application $l\mapsto\mathrm{rot}_{\theta}(q)$ is a bijection from $[0,2\pi)$ to $\gamma$.}

\end{lemma}

\begin{proof} This is an immediate consequence of the additivity of a measure and of item \textbf{(c)} from Definition \ref{genmeasure}.

\end{proof}

We highlight an interesting fact: The standard Euclidean rotation group can be obtained in any Minkowski plane. We just have to consider the group $\mathcal{R}(\gamma,\mu,o)$ where $\gamma$ is the \emph{L\"{o}wner ellipse}, which is defined as the ellipse of maximal volume contained in $B$, and $\mu$ is the measure given by twice the area of its sectors. A proof of the existence of the L\"{o}wner ellipse can be found in \cite{thompson}.\\

Next we give two examples of general rotations in the Euclidean plane. The first one relies on an area-based measure for an ellipse, which is clearly well defined. In the second we use the arc-length measure referring to a nephroid.

\begin{example}
\begin{figure}
\centering
\includegraphics{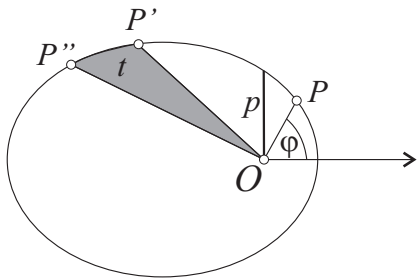}
\caption{Area-based rotation and the Kepler's model}
\label{heliofig}
\end{figure}
Consider the Euclidean plane and the system of ellipses with common focus at the origin $O$ and with major axis on the $x$-axis of the coordinate system, such that the positive half-line of $x$ contains the closest point of the ellipse (see Fig. \ref{heliofig}). In that polar coordinate system (which is called the heliocentric coordinate system for the ellipse), for which the ray $\varphi=0$ is the positive half axis $x$, we can write the radial function $r(\varphi)$ of the ellipse $G$ by the formula
$$
r(\varphi)=\frac{p}{1+\varepsilon\, \cos \varphi},
$$
where $p$ is the semi-latus rectum of the ellipse and $\varepsilon$ is the eccentricity of it, respectively. Let $\mu((\varphi',\varphi'')\angle)$ be the area of the sector enclosed by $\varphi'$, $\varphi''$, and $G$ be the arc between these lines. Hence
$$
\mu((\varphi',\varphi'')\angle)=\frac{1}{2}\int\limits_{\varphi'}^{\varphi''}\left(\frac{p}{1+\varepsilon\, \cos \varphi}\right)^2\mathrm{d}\varphi.
$$
With respect to $\mu$ and $G$ from above, for every real number $0\leq t\leq 2\pi$ there is a generalized rotation of the Euclidean plane about $O$ with this angle $t$. By Kepler's second law about planetary motions, the angle $t$ of a generalized rotation is proportional to the time of the motion of the planet. Hence the generalized rotation with angle $t$ maps the current position $P'$ of the planet to that point $P''$ of the orbit where the planet arrives after time $t$.
\end{example}

The principle of measuring the angle proportional to the area of the sector intersected by the angle domain from the basic disk $\left(G\cup \mathrm{int} G\right)$ works in all Minkowski planes and for all basic curves $G$. Note that in the Euclidean plane with the unit circle as basic curve, this choice of $\mu$ gives the usual angle measure, and that we get the usual rotations as generalized rotations by choosing $P$ to be the origin $O$. An advantage of this choice is affine invariance, but there is also a big disadvantage. Namely, the length of the arc $G$ containing the domain of the angle cannot be calculated easily from this angle measure. (As a known example, we note that the calculation of the arc-length of an ellipse leads to a complete elliptic integral of second kind, which has no closed-form solution in terms of elementary functions.) In this paper we have to create tools for the so-called \emph{rolling process}, which is a type of motion that combines rotation and translation of an object with respect to a given curve. More precisely, we combine two curves
such that they are in contact with each other without sliding (no friction). Hence we have to compare the angle of rotations of the two curves by the fact that the swept arc-lengths do agree in the time of the moving. This requires a nice connection between the angle of the generalized rotation and the corresponding arc-length of the basic curve $G$.

\begin{example}
Consider again the Euclidean plane with a cartesian coordinate system, and let $G$ be the nephroid of the unit circle with cusps on the $x$-axis, and $P$ be the origin. We define the \textit{nephroid} as an epi-cycloid created when a circle with diameter $1$ rolls on the unit circle (see Fig. \ref{nephroidfig}). It is easy to see that the parametric equation $r(t)$ of it is
\begin{equation}\label{nephroid}
r(t)=\left(\begin{array}{c}
x \\
y
\end{array}\right)=\frac{1}{2}\left(\begin{array}{c}-3\cos t+\cos 3t \\
-3\sin t+\sin 3t
\end{array}
\right)\,,
\quad 0\leq t\leq 2\pi\,,
\end{equation}
where $(t,1)$ are the polar coordinates of the point $Q\in S$. Denote by $R$ that point of the $x$-axis for which $QR$ is the common tangent of the two circles at $Q$. If $X$ is the second intersection point of the half-line $\overrightarrow{OQ}$ by the rolling circle, then the line $XR$ intersects a point $P$ of the nephroid from the rolling circle. In Figure \ref{nephroidfig} we can see the construction of two points $P_1$ and $P_2$, respectively. One of the curiosities of the nephroid is that there is a closed form to its arc-length function on the upper coordinate half plane. The length of the arc containing the points with parameters between the values $0\leq t_1<t_2\leq \pi$ is equal to $3(\sin t_2-\sin t_1)$. The generalized rotation at the origin with respect to the nephroid (and its arc-length based angle measure) sends the ray $\overrightarrow{OP_1}$ to the ray $\overrightarrow{OP_2}$, with angle measure
$$
\varphi:=\mu(OP_1,OP_2)=3(\cos t_1-\cos t_2).
$$
Hence the three-fold distance of the vertical segments $T_iQ_i$ for $i=1,2$ represents the absolute value of the angle of the rays $\overrightarrow{OP_1}$, $\overrightarrow{OP_2}$. (Thus the points $T_i$ are on the $x$-axis, respectively.)
Hence we can construct the rotated image of any point $P$ of the nephroid as follows:
\begin{figure}\label{nephroidfig}
 \centering
  \includegraphics[]{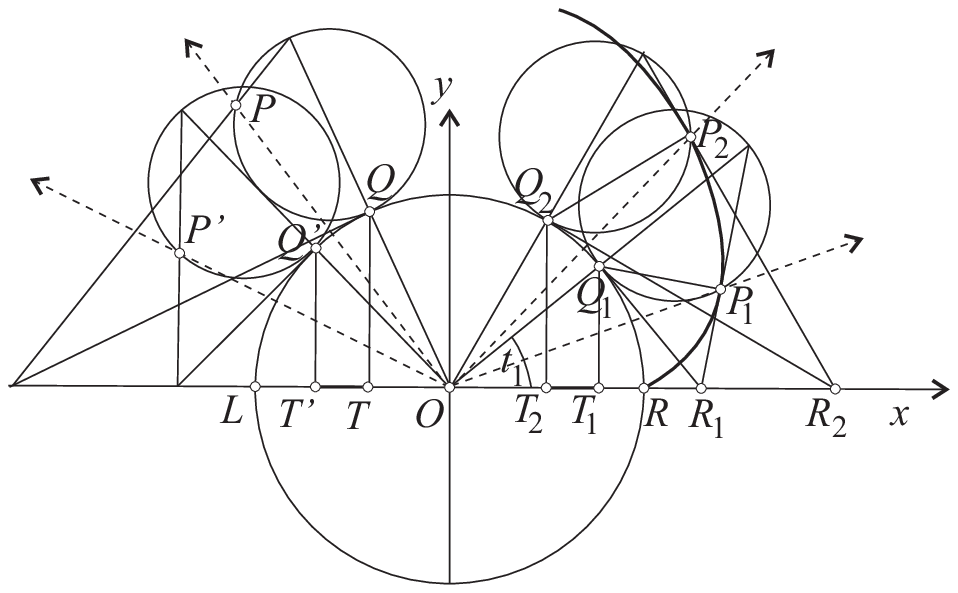}\\
  \caption{Arc-length based rotation with respect to a nephroid}
\label{nephroidfig}
\end{figure}
Assume that the point $P$ is on the upper half of the nephroid. By the intersection of the $x$-axis and a vertical line through the point $Q$ we determine the point $T$, consider the directed segment $T_1T_2$ and mark it off from $T$ on the $x$-axis. If the obtained point $T'$ is on the horizontal diameter $LR$ of the unit circle, then we can determine that point $Q'$ from the unit circle which is above $T'$ and corresponds with the searched point $P'$. If $T'$ is not on the diameter, then we mark off that outer subsegment $\overline{LT'}$ of $\overline{TT'}$ from $L$ in the relative interior of $LR$, and denote the obtained point of $\overline{LR}$ by $T'$. In this case our construction gives an image point which is on the upper half of the nephroid. It is obvious that, analogously, this construction can also be extended to the lower half of the nephroid.
\end{example}

The standard angle in the Euclidean plane can be obtained by considering arc-lengths in the unit circle, and hence the angle theory can be given in terms of the Euclidean norm. Of course, this can be carried over to Minkowski planes, and the general rotations given by the arc-length measure are possibly the most natural rotations in normed planes. We head now to take a better look at this particular case. We denote by $l$ the Minkowski arc-length of a curve defined in the usual way: as the supremum of the sums of the lengths of the polygonal approximations of $\gamma$. Let $\gamma \in (X,||\cdot||)$ be a closed rectifiable Jordan curve starlike with respect to an inner point $p$, and denote by $\mu_l$ the normalized Minkowski arc-length measure in $\gamma$. Formally, if $q_1,q_2 \in \gamma$, then

\begin{align*}\mu_l(\mathrm{arc}_{\gamma}(q_1,q_2)) = 2\pi\frac{l(\mathrm{arc}_{\gamma}(q_1,q_2))}{l(\gamma)}.\end{align*}

Of course, $\mu_l$ is a generalized measure in the sense of Definition \ref{genmeasure}. Since the measure $\mu_l$ is induced by the geometry of the plane rather than being inherent to $\gamma$, one may wonder how  the group $\mathcal{R}(\gamma,\mu_l,p)$ does rely on the initial $\gamma$ and $p$ that we have chosen. For example, in the Euclidean plane we can obtain the standard angle measure by considering the arc-length measure in any homothet of the unit circle and doing the usual normalization. Our next lemma shows that this is also true for arbitrary Minkowski planes.

\begin{lemma} Let $\gamma \in X$ be a closed rectifiable Jordan curve starlike with respect to an inner point $p$, and let $\mu_l$ be the (normalized) Minkowskian arc-length measure. Given $\alpha > 0$, denote by $\gamma_{\alpha,p}$ the curve $p + \alpha(\gamma -p)$ homothetical to $\gamma$. Then $\mathcal{R}(\gamma,\mu_l,p) = \mathcal{R}(\gamma_{\alpha,p},\mu_l,p)$.
\end{lemma}

\begin{proof} It is enough to prove that any rotation of $\mathcal{R}(\gamma,\mu_l,p)$ preserves the length of the arcs of the homothetical curves $\gamma_{\alpha,p}$. If $\gamma$ is smooth, then we can consider a regular parameterization $\gamma(t):J\subseteq\mathbb{R}\rightarrow \gamma$ of $\gamma$ and the associated parametrization of $\gamma_{\alpha,p}$ given by $\gamma_{\alpha,p}(t)=p+\alpha(\gamma(t)-p)$. Thus we can write
\begin{align*} l\left(\mathrm{arc}_{\gamma_{\alpha,p}}(q_1^*,q_2^*)\right) = \int_{t_1}^{t_2}||\gamma'_{\alpha,p}(t)||dt = \alpha\int_{t_1}^{t_2}||\gamma'(t)||dt = \alpha l\left(\mathrm{arc}_{\gamma}(q_1,q_2) \right), \end{align*}
where $q_1,q_2 \in \gamma$ are arbitrary points, and $q_1^*$ and $q_2^*$ are their respective images in $\gamma_{\alpha,p}$ by the considered homothety.

If $\gamma$ is not smooth, we just have to notice that every polygonal approximation of $\gamma$ can be obtained homothetically for $\gamma_{\alpha,p}$.

\end{proof}

Despite having the good property shown above, the arc-length rotations are not at all linear transformations. For this reason we may face some difficulties when trying to derive closed formulas for them. But we have some exceptions. Next we give an example for the Minkowski arc-length rotation which coincides with an usual Euclidean rotation.

\begin{example} Consider the norm $||\cdot||_{\infty}$ defined in $\mathbb{R}^2$ to be $||(x,y)||_{\infty} = \max\{|x|,|y|\}$. The general rotation $\mathrm{rot}_{\frac{\pi}{2}}:X\rightarrow X$ given by the Minkowski arc-length measure in the unit circle, and with respect to the origin, coincides with the usual Euclidean rotation of angle $\frac{\pi}{2}$. Indeed, the unit circle $B$ of $\left(\mathbb{R}^2,||\cdot||_{\infty}\right)$ is the square with vertices $\{(\pm1,\pm1)\}$ which, for the sake of simplicity of the used notation, we may denote in the counterclockwise way by $v_1,v_2,v_3$, and $v_4$. If $v \in [v_1,v_2]$, then $\mathrm{rot}_{\frac{\pi}{2}}$ clearly maps $v$ to the point $w$ of the segment $[v_2,v_3]$ for which $||w - v_3|| = ||v-v_2||$ (see Figure \ref{figgenrot}).

\end{example}

\begin{figure}[h]
\centering
\includegraphics{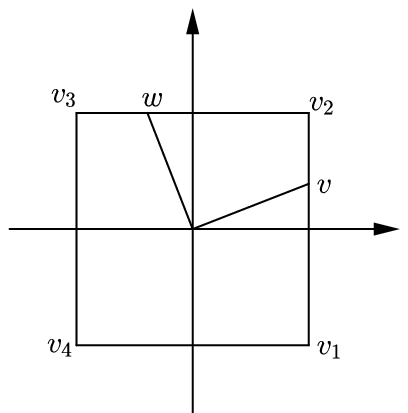}
\caption{$\mathrm{rot}_{\frac{\pi}{2}}(v) = w$}
\label{figgenrot}
\end{figure}

Translations are a simple kind of motion in Minkowski planes, and they are clearly isometries. The general rotations can also be seen as motions in the Minkowski plane, which are not necessarily isometries. Thus, we may consider the composition of translations and general rotations to obtain a larger class of motions in the Minkowski plane.

\begin{defi}\normalfont Let $\mathcal{R}(\gamma,\mu,p)$ be a fixed group of general rotations, and for any $v,w \in X$ let $t_{vw}:X\rightarrow X$ denote the translation which maps $v$ to $w$, i.e., $t_{vw}(x) = x-v+w$. We define the \textit{motion group generated by $\mathcal{R}(\gamma,\mu,p)$} to be the group of applications of the form $t_{pq}\circ\mathrm{rot}\circ t_{qp}:X\rightarrow X$, where $q \in X$ and $\mathrm{rot} \in \mathcal{R}(\gamma,\mu,p)$. When there is no possibility of confusion on the group of general rotations considered here, we will denote the motion group by $\mathcal{M}_r$.
\end{defi}

\begin{remark}\label{remark:isometries} \normalfont Notice that the motion group associated to $\mathcal{R}(\partial B,\mu_l,o)$, where $\mu_l$ is, as usual, the Minkowski arc-length measure, contains all
direction-preserving isometries of the plane.
\end{remark}

To finish this section, we highlight that, up to choosing an initial point, any group of general rotations is associated to a system of polar coordinates in the Minkowski plane. Indeed, let, as usual, $\gamma$ be a closed Jordan curve starlike with respect to a point $p$, and let $\mu$ be a generalized measure in $\gamma$. Fix a point $q_0 \in \gamma$, and consider the application $v\in\gamma\mapsto\mu(q_0,v)\in[0,2\pi)$ which associates each point $v \in\gamma$ to the measure of $\mathrm{arc}_{\gamma}(q_0,v)$ taken counterclockwise from $q_0$. Hence any point $q \in X\setminus\{p\}$ is uniquely determined by the coordinates
\begin{align*} \left(||q-p||,\mu\left(q_0,\mathrm{rad}_{\gamma,p}\left(\left.[p,q\right>\right)\right)\right),
\end{align*}
where we recall that $\mathrm{rad}_{\gamma,p}:\mathcal{R}(p)\rightarrow\gamma$ is the usual \emph{radial function} with respect to $\gamma$ and $p$. Polar coordinates can be very useful to describe the images of points of the plane by a general rotation. Indeed, if $\mathrm{rot}_{\theta} \in \mathcal{R}(\gamma,\mu,p)$, then $\mathrm{rot}_{\theta}(q)$ is clearly given by the coordinates
\begin{align*} \left(||q-p||,\mu\left(q_0, \mathrm{rad}_{\gamma,p}\left(\left.[p,q\right>\right)\right)\oplus\theta\right),
\end{align*}
where, again, the symbol $\oplus$ denotes the sum modulo $2\pi$. Observe that the system of polar coordinates given by the Minkowski arc-length measure in the unit circle and with respect to the origin is a very natural concept of polar coordinates that does only rely on the norm of the plane.

\section{Motions of rigid systems in the Euclidean plane}

Consider a plane $\Sigma'$ which is moving on the fixed plane $\Sigma $. The two simplest possibilities for such movements are given by translation and rotation. In Euclidean geometry we can substitute the planes with cartesian coordinate frames $Oxy$ and $O'uv$. When we would like to describe the motion of a point $P$ of the moving plane, we need the coordinates $u,v$ of the point $P$ in the moving frame, the coordinates $p,q$ of $O'$ in the fixed coordinate system, and the angle $\varphi $ of the positive half of the $X$-axis of the fixed frame with the positive half of the $x$-axis of the moving frame. We get the coordinates $x$, $y$ of the point $P$ in the fixed system by
$$
x=p+u\, \cos \varphi -v\, \sin \varphi\,, \qquad y=q+u\, \sin \varphi +v\, \cos \varphi.
$$
Here $p,q,\varphi$ are functions of a quantity $t$ which determines the motion. (For example, $t$ can denote the time, or any other metric parameter.) Assume that $\varphi (t)$ is not zero on an interval of $t$. Then it can be inverted, and $p,q$ can also be considered as a function of $\varphi$. (This assumption says that our motion cannot contain translations in that domain. We call such a motion \emph{non-translative planar motion}.) The derivative of the coordinate functions with respect to $\varphi$ gives the coordinates of the velocity vector of the point $P$.
It is more convenient to use vector equality, and hence we introduce some further notion. Let
$$
\R(\varphi)=\left(\begin{array}{cc}
                \cos \varphi & -\sin \varphi \\
                \sin \varphi & \cos \varphi
                \end{array}
            \right)
$$
denote the rotation about the origin with signed angle $\varphi$. Then the first equation array has the form
\begin{equation}\label{movingequation}
{\bf x}={\bf p}+\R(\varphi){\bf u}\,.
\end{equation}
If $\Q=\R(\pi/2)$ denotes the rotation with $\pi/2$, we have the following rules:
\begin{equation}\label{eq:onQ}
\Q^2=-\E, \quad \Q^3=\Q^{-1}=\overline{\Q}=-\Q, \quad \Q^4=\E,
\end{equation}
where $\E$ is the unit matrix. We denote by dot the \emph{derivative with respect to} $\varphi$, which means in this section the Euclidean arc-length parameter. It is clear that
\begin{equation}\label{eq:onR}
\dot{\R}=\Q\R, \quad \dot{(\R^{-1})}=-\Q\R.
\end{equation}
For every value of $\varphi$ there is precisely one point ${\bf u}_0$ of the moving plane for which the velocity vector vanishes. This is
$$
{\bf u}_0=\Q\R^{-1}\dot{\bf p}.
$$
This point ${\bf u}_0$ of the moving plane is a so-called \emph{instantaneous center ({\rm or} instantaneous pole)} of the motion, and the set of these points is the \emph{moving polode (centroid)}, or curve $\gamma'$ of instantaneous poles, of the moving plane. The points of the moving polode can also be obtained as rest in the frame. These points ${\bf x}_0$ are described by
$$
{\bf x}_0={\bf p}+\R{\bf u}_0={\bf p}+\Q\dot{{\bf p}}\,.
$$
 They form the so-called  \emph{fixed polode (centroid)}, or curve $\gamma $ of instantaneous centers, in the fixed plane. We examine the motion with respect to the point ${\bf x}_0$. If ${\bf x}$ is arbitrary, then ${\bf x}-{\bf x}_0=\R{\bf u}-{\bf Q}\dot{\bf p}$, and using the equality  $\dot{\bf x}=\dot{\bf p}+\Q\R{\bf u}$, we have $\Q \dot{\bf x}=\Q \dot{\bf p}+\Q\R{\bf u}$. Since ${\bf x}-{\bf x}_0=\R {\bf u}-\Q\dot{\bf p}$, we get that
$$
\dot{\bf x}=\Q({\bf x}-{\bf x}_0).
$$
Hence the velocity vector of the motion at the point ${\bf x}$ is orthogonal to the position vector from ${\bf x}_0$ to ${\bf x}$. This implies that the moving system in the given moment is a rotation about the center ${\bf x}_0$. Observe that the velocity vectors of the two polodes at their common point agree; in fact,
$$
\dot{{\bf u}}_0=\dot{\Q\R^{-1}\dot{\bf p}}=\R^{-1}\dot{\bf p}+\Q\R^{-1}\ddot{\bf p}=\dot{\bf x}_0.
$$
Hence the arc-length elements of the two curves agree, and we get that in every moment the two curves are touching. Also we see that their arc-lengths calculated from a point $\varphi_0$ to the point $\varphi $ have the same value. Hence the moving polode $\gamma'$ \emph{rolls without slipping} (or without friction) on the fixed polode $\gamma$, and this is the only rolling process which corresponds to the given motion of the planes. Hence we see the fact \emph{that every non-translatory planar motion of a rigid mechanical system in the plane can be considered as the rolling process of a curve rigidly connected with the system on a fixed curve in the plane}.
This motivates the so-called main theorem of planar kinematics, namely
\begin{theorem}\label{thm:euclideanmain}
At every moment, any constrained non-translatory planar motion can be approximated
(up to the first derivative) by an \emph{instantaneous rotation}. The
center of this rotation is called the \emph{instantaneous pole}. Thus, for each
position of the moving plane, we generally have exactly one point with velocity zero (as a
result of that, the instantaneous pole is also called velocity center).
\end{theorem}
This theorem leads to an interesting class of curves in the Euclidean plane.
\begin{defi}\label{def:roulette}
Given a curve $\gamma'$ associated with a plane $\Sigma'$ which is moving so that the curve rolls, without friction, along a given curve $\gamma $ associated with a fixed plane $\Sigma$ and occupying the same space. Then a point $P$ attached to $\Sigma'$ describes a curve in $\Sigma $ called a \emph{roulette}.
\end{defi}

Based on this rolling process we can rewrite the definition of the motion of rigid systems. Observe that every planar motion implies the motion of all points of the moving plane with respect to the fixed one. These orbits are said to be roulettes (see Definition 5). Thus, for the studied motion we consider two curves, also called \emph{polodes}, and a suitable rolling process to determine the motion of a singular point. For this purpose a method is needed to determine the fixed position of the point $P$ with respect to the moving polode. A usual method is to give a line through the point $P$ which intersects the moving polode in the point $Q$ and fixes the distance of $P$ and $Q$ and the angle of the line $PQ$ with the tangent line $t_Q$ of the moving polode at $Q$. Hence the choice of $Q$ on the moving polode is arbitrary. Fix $Q={\bf w}(0)$ and $P={\bf x}(0)$.  The points of the roulette ${\bf w}(s)$ of $Q$ can be obtained by the composition of the following transformations: translate the point $\gamma'(s)$ into the origin, rotate the image of the point of $\gamma(0)$ about the origin by the angle $\varphi(s)=\left( \dot{\gamma}(s),\dot{\gamma'}(s)\right)\angle $, and translate the obtained point by $\gamma(s)$. Hence the roulette of $Q$ in the fixed system is given by
$$
{\bf w}(s)=\R(\varphi(s))(-\gamma'(s))+\gamma(s)=\gamma(s)-\R(\varphi(s))(\gamma'(s)).
$$
Since the roulette ${\bf x}(s)$ of the point $P$ can be described by the formula ${\bf x}(s)={\bf w}(s)+\R(\varphi(s)){\bf p}$, we get
\begin{equation}\label{equatofP}
{\bf x}(s)=\gamma(s)+\R(\varphi(s))\left({\bf p}-\gamma'(s)\right).
\end{equation}
This means that if we have two touching arcs $\gamma (s)$ and $\gamma'(s)$  of a plane $\Sigma $, and we associate to the second arc a moving plane $\Sigma '$ in which its position is fixed, then the rolling process of $\gamma'(s)$ on $\gamma(s)$ (locally) determines an orbit of every point of $\Sigma'$ in a unique way. In the Euclidean plane, (\ref{equatofP}) shows that in every moment with respect to varying $p$ we have an isometry. Hence the rolling process of the arcs determines a rigid motion of the plane $\Sigma'$. This representation is locally unique, since a rigid motion uniquely determines its polodes. Hence we have
\begin{theorem}\label{thm:conveuclideanmain}
If $\gamma,\gamma':[0,\beta]\rightarrow \mathbb{R}^2$ are two simple Jordan arcs with common touching point $\gamma(0)=\gamma'(0)$ such that $s$ is the arc-length parameter of both of them (considered from the points $\gamma(0), \gamma'(0)$ to the points $\gamma(s)$, $\gamma'(s)$, respectively), then for every $s\in [0,\beta]$ we have an isometry $\Phi_s$ sending the original position vector ${\bf p}$ into the instantaneously position $\Phi_s(p)$. If $\gamma$ and $\gamma'$  have, for all $s\in [0,\beta]$, unique tangents at their points $\gamma(s)$ and $\gamma'(s)$, respectively, then, for all $s\in[0,\beta]$, $\Phi_s$ is uniquely determined and can be described by the vector equation
$$
\Phi_s({\bf p})=\gamma(s)+\R(\left( \dot{\gamma}(s),\dot{\gamma'}(s)\right)\angle)\left({\bf p}-\gamma'(s)\right)\,.
$$
Here $\dot{\gamma}(s)$ and $\dot{\gamma'}(s)$ denote the unit tangent vectors at $\gamma(s)$ and $\gamma'(s)$, respectively, and $\R(\theta)$ is the rotation with the angle $\theta$. For fixed ${\bf p}$, the graph of the function $\Phi_{(\cdot)}({\bf p}):[0,\beta]\rightarrow \Sigma $ is said to be the \emph{roulette} of the point $P={\bf p}\in \Sigma$ for the rigid motion given by the system of isometries $\{\Phi_s:\, s\in[0,\beta]\}$.
\end{theorem}

\section{Flexible motions of a Minkowski plane}

Our purpose now is to extend Theorem \ref{thm:conveuclideanmain} to Minkowski planes. For this purpose we defined already the motion group $\mathcal{M}_r$ of the Minkowski plane, which is a good analogoue of a motion group of the Euclidean plane. Clearly, we have to omit the condition that a motion is an isometry, due to the smallness of the actual isometry group in a Minkowski plane. Of course, any motion group $\mathcal{M}_r$ contains all the translations. On the other hand, it is possible that the image of a metrical segment under a general rotation is not a metrical segment. Hence the concept of Euclidean rigid motions has to be redefined.

\subsection{Notions and background}

We concentrate on Theorem \ref{thm:conveuclideanmain} for the Euclidean planar motions, and we will consider from now on that the motion group $\mathcal{M}_r$ is the motion group associated with the group of general rotations $\mathcal{R}(\partial B,\mu_l,o)$. In other words, we will consider the rotations by arc-length of the unit circle with respect to the origin.

\begin{defi}\label{rolling}
The rectifiable Jordan curve $\gamma'(s)$ \emph{rolls without slipping} on the rectifiable Jordan curve $\gamma (s) $ if in every moment $s\in[0,\beta]$ the two curves touch each other, and the respective arc-lengths calculated from their common point $\gamma(0)=\gamma'(0)$ to the other one $\gamma(s)=\gamma'(s)$ are equal to each other and also to the common parameter $s$.
\end{defi}

Having the rolling procedure and the motion group $\mathcal{M}_r$, we can define the continuous (but not rigid) motions of a Minkowski plane. Assume that in this section any considered curve is a rectifiable Jordan curve, with unique tangent at all of its points, respectively. We denote the unit tangent vector of $\gamma$ at its point $\gamma(s)$ by $\dot{\gamma}(s)$. (Since $s$ means the arc-length parameter, this notation corresponds to the usual Euclidean notation based on the arc-length derivative of the position vector.)

\begin{defi}
If the rectifiable Jordan curve $\gamma'(s)$ rolls, without slipping, on the rectifiable Jordan curve $\gamma(s)$, then we define \emph{the flexible motion corresponding to the rolling curves} $\gamma$ and $\gamma'$ as the following set of mappings:
\begin{equation}\label{motionequality}
\{\Phi_s({\bf p})=\gamma(s)+\R(\varphi_s)\left({\bf p}-\gamma'(s)\right):\, s\in[0,\beta]\},
\end{equation}
where $R(\varphi_s)\in \mathcal{R}(\partial B,\mu_l,o)$ denotes the general rotation which maps the (oriented) direction $\dot{\gamma}(s)$ to the (also oriented) direction $\dot{\gamma}'(s)$. A curve given by the graph of a fixed point ${\bf p}=P$ is called the \emph{roulette} of $P$.
\end{defi}

We can provide a simple illustrative example here: let $\gamma(s)$ be an arc-length parameterization of $\partial B$ starting at an arbitrary $p \in \partial B$. For any natural number $n \geq 2$ one can set $\gamma'(s) = \frac{n-1}{n}p + \frac{1}{n}\gamma(ns)$. Then, $\gamma'(s)$ is an inner circle which rolls without slipping on the unit circle $\gamma(s)$, and each of its points clearly describe a curve with $n$ cusps (see Figure \ref{fighypocycloid}). This can be regarded as an analogue to the standard hypocycloids of the Euclidean plane.

\begin{figure}[h]
\centering
\includegraphics{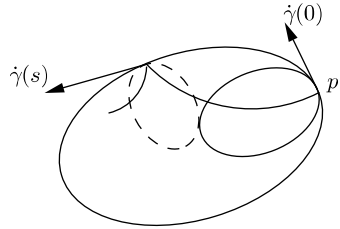}
\caption{A Minkowski hypocicloid}
\label{fighypocycloid}
\end{figure}

The vector
$$
\frac{\partial{\R}(\varphi)}{\partial \varphi}(\bf{x})=\lim\limits_{\varepsilon\rightarrow 0}\frac{\R(\varphi+\varepsilon)(\bf x)-\R(\varphi)(\bf x)}{\varepsilon}
$$
is the tangent vector of $|\bf x|\partial B$ at the point $\bf x$. This means that $\frac{\partial{\R}(\varphi)}{\partial \varphi}(\bf{x})$ is Birkhoff normal to the vector ${\R}(\varphi)(\bf{x})$. (For relations between semi-inner products and Birkhoff orthogonality, see, e.g., \cite{gho 1} or \cite{gho 2}.) Denote by $\Q$ that mapping which sends the vectors to their Birkhoff normals with the same norm, and by $\Q^{-1}$ the mapping which sends the vectors to their Birkhoff transversals with the same lengths. (Note that Birkhoff orthogonality is not a symmetric relation; see, e.g., \cite{martini-swanepoel 1} or \cite{mcshane}. So, in general, if $x$ is Birkhoff normal to $y$, then $y$ not to $x$. However, we have a possibility to ``reverse'' the formulation "$x$ is Birkhoff normal to $y$". We say in this case that $y$ is transversal to $x$.) Since the tangent vector of the roulette of $P$ at the point with parameter $s$ is
$$
\dot{\Phi_s}({\bf p})=\dot{\gamma}(s)+\Q({\R}(\varphi(s))({\bf p}-\gamma'(s))\dot{\varphi}(s) - R(\varphi_s)\dot{\gamma}'(s) = \Q({\R}(\varphi(s))({\bf p}-\gamma'(s))\dot{\varphi}(s) ,
$$
we get, with the semi-inner product defined by the Minkowski norm, that
$$
\left[\dot{\Phi_s}({\bf p}),\Phi_s({\bf p})-\gamma(s) \right]=0.
$$
Hence we obtain
\begin{statement}\label{st:instpole}
The velocity vector of the flexible motion of a point $\Phi_s({\bf p})$ of the roulette in a moment $s$ is Birkhoff normal to that vector $\Phi_s({\bf p})-\gamma(s)$ which shows from the point to the instantaneous pole of the motion.
\end{statement}

From Statement \ref{st:instpole} we can see that our definition yields the same kinematics in the Minkowski plane as given by usual motions of rigid systems in the Euclidean plane.

\subsection{Curvature and the Euler-Savary equations (introducing remarks)}

We will prove now the so-called \emph{Euler-Savary equations} (see \cite{rattan}) for normed planes. In space-time (or in the Minkowski plane with indefinite scalar product) this was investigated by Ikawa \cite{ikawa}. He defined roulettes and proved the Euler-Savary equations for normed planes, with respect to this semi-Riemannian geometry of constant curvature. Because of the rich isometry group of this plane, the validity of these results is not so surprising as in our case.

In this section we have to assume second order differentiability of the unit circle, and we have to introduce the concepts of curvature and curvature radius of a curve, respectively. Fortunately, in Minkowski planes several such concepts are known. Curvatures for curves in Finsler spaces were introduced for dimension $n=2$ by
Underhill \cite{underhill} and Landsberg \cite{landsberg}. For general $n$ they were introduced by Finsler \cite{finsler,finsler 1}. The definitions coincide for $n=2$. The underlying idea of these definitions is this: If $\gamma(s)$ is a curve with tangent $t$ at a given point $q$, then the line parallel to this tangent through the origin intersects the unit circle in a point $q'$ (in fact, in a pair of points, but it will not matter which point is chosen).
There is exactly one ellipsoid with the origin as center through $q'$ which has at $q'$ the same second differential as the unit circle. This ellipsoid determines a Euclidean metric $E(q)$. Finsler defines the curvatures of $\gamma(s)$ at $q$ as the curvatures at $q$
of $\gamma(s)$ as a curve in $E(q)$. Obviously, $E(q)$ exists only if the unit circle has a second
differential at $q '$ and the indicatrix is a non-degenerate ellipse. Actually, this idea is significant only if $C$ is of class $C^2$ and has positive Gauss curvature. Thus $\gamma(s)$ may not even have a curvature when it is analytic.

There exists another definition of curvature for curves in general spaces
which is due to Menger \cite{menger} (for modifications of this concept see \cite{haantjes}).
Haantjes' curvature coincides with that of Finsler. Hence Haantjes' main result in \cite{haantjes} means that, in Minkowski spaces, Menger's definition
coincides  with Finsler's definition.

\subsection{Busemann curvature and the general sine function of Busemann}

In \cite{busemann 2}, Busemann gave another concept of curvature. Before discussing it, we will define Busemann's sine function $\mathrm{sm}:\mathcal{L}\times \mathcal{L}\rightarrow \mathbb{R}$ from the pairs of lines to the field of reals. If $a,b\in \mathcal{L}$ and $s_a$, $s_b$ are two segments on these lines, respectively, then we can define the parallelogram $\pi(s_a,s_b)$ that is spanned by $s_a$ and $s_b$. If we write  $\mathrm{area}(\pi(s_a,s_b))$ for the Busemann area of $\pi(s_a,s_b)$ and take into consideration the Minkowski lengths $|s_a|$, $|s_b|$ of $s_a$ and $s_b$, then the \emph{Minkowski sine function of Busemann} can be defined as follows:
\begin{equation}\label{def:sm}
\mathrm{sm}(a,b):=\frac{\mathrm{area}(\pi(s_a,s_b))}{\|s_a\|\|s_b\|}\,.
\end{equation}
From the definitions of Minkowski length and Minkowski area it follows that $sm(a,b)$ is not depending on the segments $s_a$ and $s_b$. Thus, it depends only on the lines $a$, $b$.

In $n$-dimensional Minkowski space let $\gamma(s)$ be a curve which is, in the Euclidean sense, of class $C^r$ and parametrized by the Minkowskian arc-length $s$. Let $\gamma(s_i)$, $i=0,1,\ldots,n$, be $n+1$
points on $\gamma (s)$. Let $T_r$ denote the $r$-dimensional Minkowski volume of the $r$-dimensional simplex that is spanned by the points $\gamma(s_i)$,  $i=0,1\ldots r$. Then we define the \emph{$(r-1)$-{th} curvature} $\chi_{r-1}$ of the curve $\gamma $ in its point $\gamma(s)$ by the limit
\begin{equation}\label{def:curveture}
\chi_{r-1}(s)=\frac{r^2}{r-1}\lim\limits_{s_i\rightarrow s}\frac{1}{\|\gamma(s_r)-\gamma(s_0)\|}\frac{T_rT_{r-2}}{T_{r-1}T^\star_{r-1}}
\end{equation}
(see \cite{busemann 2}), where $T^\star_{r-1}$ denotes the volume of the $(r-1)$-dimensional simplex spanned by the points $\gamma(s_i)$, $i=1,\ldots r$. Let $D_r$ be the following quantity:
$$
D_r(s)=r!\prod\limits_{i=1}^{r}i!\lim\limits_{s_i\rightarrow s}\frac{T_r}{\prod\limits_{i<j}\|\gamma(s_i)-\gamma(s_j)\|}\,.
$$
Then for $D_{r-2}(s)\ne 0$ we get the following form of the curvature function:
$$
\chi_{r-1}(s)=\frac{D_r(s)D_{r-2}(s)}{D_{r-1}^2(s)}.
$$
This formula can be rewritten by the concept of the general sine function of two flats of the $n$-dimensional Minkowski space, but we need only the case of dimension $2$. Hence, using that $D_0(s)=1$, the curvature is
\begin{equation}\label{def:curvsm}
\chi_\gamma(s):=\chi_1(s)=\frac{D_2(s)}{D_1^2(s)}=2\lim\limits_{s_0,s_1,s_2\rightarrow s}\frac{\mathrm{sm}(g(\gamma(s_0),\gamma(s_1)),g(\gamma(s_1),\gamma(s_2)))}{\|\gamma(s_2)-\gamma(s_0)\|},
\end{equation}
where $g(x,y)$ denotes the line through $x$ and $y$.

There is a nice connection between the concepts of curvature given by  Finsler and Busemann. In a Minkowski plane, the Finsler curvature $\chi^f$ and the curvature $\chi$ of Busemann of a curve $\gamma (s) $
at a point $P$, with position vector $\overline{p}$, are related by
$$
(\chi^f(P))^2=\frac{\chi^2(P)}{\chi_T(\overline{p})}\,,
$$
where $\chi_T(\overline{p})$ is the curvature of the isoperimetrix (see \cite{busemann 1}) at a point $\overline{p}$
(the tangent of the isoperimetrix has to be parallel to the tangent of $\gamma(s)$ at $p$).

A curve $\gamma(s)$ having curvature in Euclidean sense has also curvature in the sense of Busemann. These two curvatures can be compared. For this purpose we have to use the $\sigma$-function introduced by Busemann. Let $V_r$ be an $r$-flat of a Minkowski space of dimension $n$. If $U(V_r)$ is the set in which the $r$-flat, parallel to $V_r$ and passing
through the origin, intersects the solid Minkowskian unit sphere, then we define $\sigma(V_r)$ as the ratio of the $r$-dimensional volume of the $r$-dimensional unit ball and the Euclidean volume of $U(V_r)$. Observe that if $\gamma(s)$ is a $C^1$ curve with tangent line $t_P$ and velocity vector $\dot{\gamma}(s)$ at the point $P=\gamma(s)$, then by the definition of Minkowski length we have
\begin{equation}\label{comp:tang}
\|\dot{\gamma}(s)\|=\sigma(t_P)\|\dot{\gamma}(s)\|_E,
\end{equation}
where $\|\cdot\|_E$ means the Euclidean norm. Busemann \cite{busemann 2} proved that if $\chi_E(P)$ denotes the Euclidean curvature of $\gamma(s)$ at the point $P$, $t_P$ is written for the tangent line of $\gamma(s)$ at $P$, and $T_P$ is the osculating plane of the curve at $P$, then
\begin{equation}\label{comp:curv}
\chi(P)=\frac{\sigma(T_P)}{\sigma^3(t_P)}\chi^E(P).
\end{equation}

\subsection{The generalized Euler-Savary equations and their combination}

We use these formulas to establish a close analogue to the Euler-Savary theorem on rigid motions in the Euclidean plane. First of all, we consider two curves $\gamma$ and $\gamma'$. Hence we have to use a  suitable lower subscript for the curvature function. We also have the concept of \emph{curvature radius} $r_{\gamma}$ which is, as well-known, the reciprocal value of the curvature at the given point $K=\gamma(s)$. With these notions we are able to formulate
\begin{theorem}[Second Euler-Savary equation]\label{thm:Euler-Savary}
If the unit circle of the Minkowski plane is two times continuously differentiable, then the following equality holds:
\begin{equation}\label{equ:Euler-Savary}
\chi_{\gamma}-\chi_{\gamma'}=\frac{1}{r_{\gamma}}-\frac{1}{r_{\gamma'}}=\frac{\sigma(T_K)}{\sigma^2(t_K)}\frac{1}{\alpha_K}\,.
\end{equation}
Here $r_{\gamma}$ is the curvature radius of the fixed polode at its point $K=\gamma_s$, $r_{\gamma'}$ is the curvature radius of the moving polode at its point $K=\gamma'_s$, and $\alpha_K$ is the length of the common velocity vector of the fixed and moving polodes at the moment $s$ and at the instantaneous pole $K=\gamma(s)=\gamma'(s)$.
\end{theorem}

\begin{proof}
Using (\ref{comp:tang}), (\ref{comp:curv}) and the Euclidean version of the Euler-Savary equation, we get
$$
\chi_{\gamma}-\chi_{\gamma'}=\frac{\sigma(T_K)}{\sigma^3(t_K)}\left(\chi^E_{\gamma}-\chi^E_{\gamma'}\right)=\frac{\sigma(T_K)}{\sigma^3(t_K)}\frac{1}{\alpha^E_K}=\frac{\sigma(T_K)}{\sigma^2(t_K)}\frac{1}{\alpha_K},
$$
as we stated.
\end{proof}

To prove an analogue of the first Euler-Savary equation, we need a deeper investigation of the Busemann curvature.
Let $t_K$ be the common tangent of the polodes at their common point $K$, which is the $x$-axis of a Euclidean orthogonal coordinate system $(x,y)$. We denote by $O,O'$ the curvature centers of the curves $\gamma(s)$ and $\gamma'(s)$, respectively. Then $O$ and $O'$ coincide with the line $y$ and $\chi^E_\gamma(K)=1/\|KO\|_E$, $\chi^E_{\gamma'}(K)=1/\|KO'\|_E$. Denote by $P$ any point of the moving plane corresponding to the curve $\gamma'$ with the vector ${\bf p}=\overrightarrow{KP}$. As we saw in Statement \ref{st:instpole}, the line $n_P$ of the points $K,P$ contains the Minkowskian curvature center of the roulette of $P$, since it is Birkhoff normal to the tangent $t_P$ at $P$. Denote this point by $P'$. We have at $\gamma(0)=\gamma'(0)=K$ that $\R(\varphi(0))=\mathrm{id}$, and $\dot{\gamma}(0)={\bf v}_K$, where ${\bf v}_K$ is the common (Minkowskian) velocity vector at $K$.
Hence we have the equality
$$
{\bf v}_P:=\left.\frac{\partial (\Phi_s({\bf p}))}{\partial s}\right|_0=\left.\Q({\R}(\varphi(s))({\bf p}-\gamma'(s)))\dot{\varphi}(s)\right|_0=\Q(\overrightarrow{KP})\dot{\varphi}_0.
$$
Thus, the acceleration vector $a_P$ is
$$
{\bf a}_P=\left.\frac{\partial {\bf v}_P}{\partial s}\right|_0=\lim\limits_{\varepsilon\rightarrow 0}\frac{\Q({\R}(\varphi(\epsilon)))({\bf p}-\gamma'(\epsilon))\dot{\varphi}(\epsilon)-\Q({\R}(\varphi(0))({\bf p}-\gamma'(0)))\dot{\varphi}(0)}{\varepsilon}+\Q(\overrightarrow{KP})\ddot{\varphi}(0)=
$$
$$
=\dot{\varphi}(0)\left(\lim\limits_{\varepsilon\rightarrow 0}\frac{\Q({\R}(\varphi(\epsilon)))({\bf p}-\gamma'(\epsilon))-\Q({\bf p}-\gamma'(\epsilon))}{\varepsilon}+\lim\limits_{\varepsilon\rightarrow 0}\frac{\Q({\bf p}-\gamma'(\epsilon))-\Q({\bf p}-\gamma'(0))}{\varepsilon}\right)+
$$
$$
+\Q(\overrightarrow{KP})\ddot{\varphi}(0).
$$
Observe that if $\Q$ would be an additive function and we could change it with the limit process, then the first term in the bracket could be simplified to the quantity
$\Q\Q(\overrightarrow{KP})\dot{\varphi}(0)$ and the second one is nothing else than the velocity vector of the moving polode at zero. (In our case it is also the velocity vector of the fixed polode.)
In general this is not so, since the additivity of the operation $\Q$ implies that the space is Euclidean with a standard inner product. Thus, for further investigations we need a quantity which measures the difference between  the given limits and the optimal values (attended by the case of inner product planes). This motivates the following lemma.
\begin{lemma}\label{lem:accel} Assume that $\gamma(s)$ is a curve of $C^1$ type parametrized by its arc-length. If $\textbf{a},\textbf{b},\textbf{c}\in \gamma(s)$ and $t_{\textbf{c}}$ denotes the tangent of the curve $\gamma(s)$ at its point $\textbf{c}$, then we have
\begin{equation}\label{derQ}
\lim\limits_{\textbf{a},\textbf{b}\rightarrow \textbf{c}}\frac{\Q(\textbf{b})-\Q(\textbf{a})}{\|\textbf{b}-\textbf{a}\|}=\frac{1}{\sigma(t_\textbf{c})}\Q^2(\textbf{c}).
\end{equation}
\end{lemma}
\begin{proof}
By definition the line $g(\Q(a),\Q(b))$ tends to the tangent line $t_{\Q(c)}$ of the curve $\Q(\gamma(s))$ at its point $\Q(c)$. Since it is parallel to the vector $\Q\Q(c)$ of length $\|c\|$, we have to determine only the length of the limit vector of the left hand side. But we have
$$
\frac{\|\Q(\textbf{b})-\Q(\textbf{a})\|}{\|\textbf{b}-\textbf{a}\|}=\frac{\sigma(g(\Q(\textbf{a}),\Q(\textbf{b})))\|\Q(\textbf{b})-\Q(\textbf{a})\|_E}{\sigma(g(\textbf{a},\textbf{b}))\|\textbf{b}-\textbf{a}\|_E}\,,
$$
and hence, using the continuity of the function $\sigma$, we get
$$
\lim\limits_{a,b\rightarrow c}\frac{\Q(\textbf{b})-\Q(\textbf{a})}{\|\textbf{b}-\textbf{a}\|}=\frac{\sigma(t_{\Q(\textbf{c})})}{\sigma(t_\textbf{c})}\left\|\left.\dot{\Q(\gamma(s))}\right|_{\textbf{c}}\right\|_E=\frac{\sigma(t_{\Q(\textbf{c})})}
{\sigma(t_\textbf{c})}\|\Q(\Q(\textbf{c}))\|_E=\frac{\|\Q(\Q(\textbf{c}))\|}{\sigma(t_\textbf{c})},
$$
as we stated.
\end{proof}
By Lemma \ref{lem:accel} we get an expression for the acceleration vector above, namely
$$
{\bf a}_P = \dot{\varphi}^2(0)\left(\frac{1}{\sigma(t_{P})}\Q^2(\overrightarrow{KP})- \frac{1}{\sigma(t_{K})}\Q\left(\frac{{\bf v}_K}{\dot{\varphi}(0)}\right)\right)+ \Q(\overrightarrow{KP})\ddot{\varphi}(0),
$$
where ${\bf v}_K$ means the common velocity vector of the curves $\gamma(s)$, $\gamma'(s)$ at $K=\gamma(0)=\gamma'(0)$.
\begin{figure}
 \centering
   \includegraphics[]{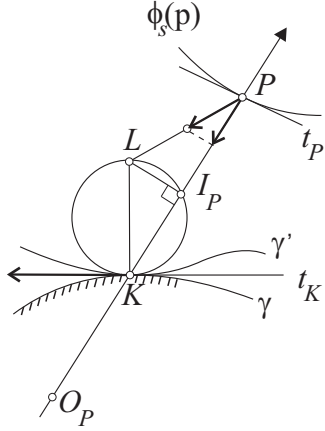}
 \caption{The point $L$}
\label{EulerSavaryfig}
\end{figure}
We now introduce a point $L$ (see Figure \ref{EulerSavaryfig}) such that
$$
\overrightarrow{LP}=-\left(\frac{1}{\sigma(t_{P})}\Q^2(\overrightarrow{KP})-\frac{1}{\sigma(t_{K})}\Q\left(\frac{{\bf v}_K}{\dot{\varphi}(0)}\right)\right)\,,
$$
hence the acceleration may be written as
\begin{equation}\label{acc}
{\bf a}_P=\ddot{\varphi}(0)\Q(\overrightarrow{KP})-\dot{\varphi}^2(0)\overrightarrow{LP}.
\end{equation}
Observe that $\Q(\overrightarrow{KP})$ is normal to the vector $\overrightarrow{KP}$, and that it has no component parallel to $\overrightarrow{KP}$. The vector $-\dot{\varphi}^2(0)\overrightarrow{LP}$ lies along $g(L,P)$ and is directed toward $L$, so its projection contributes to both components (one of them  parallel to $g(K,P)$, and the other one normal to it) of the acceleration vector. Hence a unique situation exists if $\overrightarrow{LP}$ is normal to $\overrightarrow{KP}$. In this case, the acceleration vector has no component parallel to $g(K,P)$ implying that the radius of curvature of its path is infinite.
\begin{defi}\label{def:inflectioncurve}
The locus of all points $P$  for which $\overrightarrow{LP}$ is normal to $\overrightarrow{KP}$ is the \emph{inflection curve} of the motion. The point $L$ is the \emph{inflection pole} of the motion.
\end{defi}
The inflection curve is the "Thales circle" of the segment $\overline{KL}$ with respect to Birkhoff orthogonality. We have to prove the following properties of it:
\begin{statement}
The inflection curve $\iota$ is a closed curve. It is starlike with respect to the point $K$ if the unit circle is smooth. However, in general it does not bound a convex domain.
Finally, if it is a Minkowski circle for all segments of the normed plane, then the plane is Euclidean.
\end{statement}

\begin{proof}
The first statement is trivial. To prove the starlike property, consider the notation of Fig. \ref{thalesianfig}.
\begin{figure}
  \centering
 \includegraphics[]{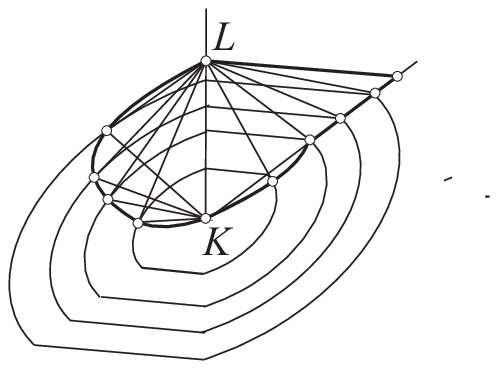}\\
 \caption{The curve of inflection}
\label{thalesianfig}
\end{figure}

First of all, observe that every half-line of the upper half-plane starting at $K$ intersects $\iota$ in a point or a segment. This follows from the fact that if for $Y'\ne Y''$, $Y',Y''\in\iota$ and $\overrightarrow{KY''}=t\overrightarrow{KY'}$, then  $[\overrightarrow{LY'},\overrightarrow{KY'}]=0$ and $t[\overrightarrow{LY''},\overrightarrow{KY'}]=[\overrightarrow{LY''},\overrightarrow{KY''}]=0$. If $\tau\in [0,1]$ is arbitrary, then we have that with the point $Y(\tau)$ holding $\overrightarrow{KY(\tau)}=(1-\tau)\overrightarrow{KY'}+\tau\overrightarrow{KY''}=(1-\tau+t\tau)\overrightarrow{KY'}= \left(\frac{1-\tau}{t}+\tau\right)\overrightarrow{KY''}$ we have also $\overrightarrow{LY(t)}=(1-\tau)\overrightarrow{LY'}+\tau\overrightarrow{LY''}$, and this implies
$$
\left[\overrightarrow{LY(\tau)},\overrightarrow{KY(\tau)}\right]= \left[(1-\tau)\overrightarrow{LY'}+\tau\overrightarrow{LY''},\overrightarrow{KY(\tau)}\right]=
$$
$$
=(1-\tau)(1-\tau+t\tau)\left[\overrightarrow{LY'},\overrightarrow{KY'}\right]+ \tau\left(\frac{1-\tau}{t}+\tau\right)\left[\overrightarrow{LY'},\overrightarrow{KY''}\right]=0.
$$
Now, if a tangent line of the unit circle $\mathcal{C}$ with center $K$ is uniquely determined  at its point $P_0$, then this tangent and the tangents of the positive homothetic copies  $t\mathcal{C}$ at $t\overrightarrow{KP_0}$ are parallel to each other. This implies that on the half-line $\overrightarrow{KP}$ there is precisely one point $t\overrightarrow{KP_0}$ at which the tangent of $t\mathcal{C}$ goes through the point $L$. This means that when the unit circle is smooth, the inflection curve is starlike with respect to the point $K$. In addition, we also proved that we can associate to a non-smooth point (vertex) of the unit circle a segment on the inflection curve lying on the corresponding half-line $\overrightarrow{KP}$. This immediately shows that for polygonal norms the domain of the inflection curve is not convex, and a counterexample to this fact can be easily seen in the smooth case.

The last statement is an easy consequence of the fact that if $x+y$ is Birkhoff orthogonal to $x-y$ for any distinct unit vectors $x,y \in X$, then $X$ is Euclidean (see \cite{alonso}). Indeed, one just has to consider the Thales circle of the segment connecting $x$ and $-x$.

\end{proof}

By the physical meaning of the acceleration vector, the absolute value of the normal component of this vector is
$$
\dot{\varphi}^2(0)\|\overrightarrow{KP}\|^2\chi(P)=\dot{\varphi}^2(0)\frac{\|\overrightarrow{KP}\|^2}{\|\overrightarrow{PO_P}\|},
$$
where $\chi(P)$ and $\|\overrightarrow{PO_P}\|$ are the curvature and the curvature radius $R_P$ of the roulette at $P$, respectively. Along the path, the direction is always normal. If this normal is oriented from $K$ to $P$, then the magnitude and orientation of the normal component of the acceleration vector may be defined in terms of real numbers, and it will be positive if $PO_P$ is positive, i.e., if it has the same orientation as $KP$. If $PO_P$ has orientation opposite to that of $KP$, it will be negative.

On the other hand, it can also be obtained from the length of the orthogonal projection of $\dot{\varphi}^2(0)\overrightarrow{PL}$ to the path normal line $g(P,K)$. Hence we have
$$
\dot{\varphi}^2(0)\frac{\|\overrightarrow{KP}\|^2}{\|\overrightarrow{PO_P}\|}=\dot{\varphi}^2(0)\left[\frac{1}{\sigma(t_{P})}\Q^2(\overrightarrow{KP})- \frac{1}{\sigma(t_{K})}\Q\left(\frac{v_K}{\dot{\varphi}(0)}\right),(\overrightarrow{KP})^0\right],
$$
with $(\overrightarrow{KP})^0$ as unit vector. Denote the second intersection point of the line $g(K,P)$ with the inflection curve by $I_P$. Then
$$
\overrightarrow{PI_P}=\frac{\|\overrightarrow{KP}\|^2}{\|\overrightarrow{PO_P}\|}(\overrightarrow{KP})^0=\left[\frac{1}{\sigma(t_{P})}\Q^2(\overrightarrow{KP})- \frac{1}{\sigma(t_{K})}\Q\left(\frac{v_K}{\dot{\varphi}(0)}\right),(\overrightarrow{KP})^0\right](\overrightarrow{KP})^0\,,
$$
and so we have the equality
$$
\frac{\|\overrightarrow{KP}\|^2}{\|\overrightarrow{O_PP}\|}=\|\overrightarrow{I_PP}\|.
$$
Hence we get the following geometric form of the first Euler-Savary theorem.
\begin{theorem}\label{thm:Euler-SavaryI}
The instantaneous center $K$ and the curvature center $O_P$ of the roulette at its point $P\ne K$ satisfy the equality
\begin{equation}\label{firstEulerSavary}
\|\overrightarrow{O_PP}\|=\frac{\|\overrightarrow{KP}\|^2}{\|\overrightarrow{I_PP}\|},
\end{equation}
where the second intersection point of the path normal line at $P$ with the inflection curve is the point $I_P$.
\end{theorem}

By the law of sine introduced earlier, $O_PP$ and $I_PP$ are always marked off in the same orientation along the line $KP$. Thus, when $I_P$ has been established, the orientation of $I_PP$ gives the orientation of $O_PP$. Hence  equality (\ref{firstEulerSavary}) has an equivalent form for directed segments (with Minkowski lengths):
\begin{equation}\label{standfES}
\frac{1}{KP}-\frac{1}{KO_P}=\frac{1}{KI_P}.
\end{equation}

From this inequality we can see immediately that the curvature radius of the point of the inflection curve is infinite. Similarly, the centers of path
curvature of all points at infinity are on the \emph{return curve} obtained as the image of the inflection curve under reflection at the point $K$. To see a connection between the two Euler-Savary equations, we give a connection between $KI_P$ and $\alpha_K$ which is the length of the common velocity vector of the fixed and moving polodes at $K$. For the sine function $\mathrm{sm}(g_1,g_2)$ of Busemann the theorem of sines holds, and it is compatible with the normality concept of Birkhoff. Hence we have
$$
\frac{\|\overrightarrow{KI_P}\|}{\|\overrightarrow{KL}\|}=\frac{\mathrm{sm}(g(K,L),g(L,I_P))}{\mathrm{sm}(g(K,I_P),g(L,I_P))}=
\frac{\sin(g(K,L),g(L,I_P))\frac{\sigma(T_K)}{\sigma(g(K,L))\sigma(g(I_P,L))}} {\sin(g(K,I_P),g(L,I_P))\frac{\sigma(T_K)}{\sigma(g(K,I_P))\sigma(g(I_P,L))}}=\sin\Psi \frac{\sigma(g(K,P))}{\sigma(g(K,L))},
$$
where $\Psi$ is the Euclidean angle between the tangent line $t_K$ at $K$ and the line $g(K,P)$. From this we get the common form of the first and second Euler-Savary equations. By
$$
\left(\frac{1}{KP}-\frac{1}{KO_P}\right)\mathrm{sm}(g(K,P),t_K)\frac{\sigma(t_K)\sigma(g(K,P))}{\sigma(T_K)}= \left(\frac{1}{KP}-\frac{1}{KO_P}\right)\sin\Psi=\frac{\sigma(g(K,L))}{\sigma(g(K,P))}\frac{1}{KL}\,,
$$
and using that the velocity vector ${\bf v}_K$ of the instantaneous pole at $K$ is equal to $V_K=\dot{s}(0)\left.\frac{\partial \gamma (s(\omega))}{\partial s}\right|_0=\alpha_K {\bf v}_K^0$, we get that the acceleration vector is ${\bf a}_K=\ddot{s}(0){\bf v}_K^0+\alpha_K {\bf n}_K^0$. This implies that its normal component is
$[{\bf n}_K^0,{\bf a}_K]{\bf n}_K^0=\alpha_K {\bf n}_K^0$. On the other hand, from the definition of the point $L$ and the continuity property of the examined curves we get that if $P$ tends to $K$,
then $\overrightarrow{LP}$ tends to
$$
\overrightarrow{LK}=\frac{1}{\sigma(t_{K})}\Q\left(\frac{{\bf v}_K}{\dot{\varphi}(0)}\right).
$$
So we have $\|\overrightarrow{LK}\|=\alpha_K/\left(\sigma(t_{K})\dot{\varphi}(0)\right) $, and if we assume that the length of the directed segment $KL$ is positive, then we get
$$
\left(\frac{1}{KP}-\frac{1}{KO_P}\right)\mathrm{sm}(g(K,P),t_K)\frac{\sigma(t_K)\sigma^2(g(K,P))}{\sigma(T_K)\sigma(g(K,L))}=\frac{1}{\|\overrightarrow{KL}\|}= \frac{\sigma(t_{K})\dot{\varphi}(0)}{\alpha_K}=
$$
$$
=\frac{\sigma(t_{K})\dot{\varphi}(0)\sigma^2(t_K)}{\sigma(T_K)}\left(\chi_{\gamma}-\chi_{\gamma'}\right)\,.
$$
This yields the \emph{combined formula of the two Euler-Savary equations}, namely
\begin{equation}\label{combES}
\left(\frac{1}{KP}-\frac{1}{KO_P}\right)\mathrm{sm}(g(K,P),t_K)\frac{\sigma^2(g(K,P))}{\sigma^2(t_K)\sigma(g(K,L))}= \dot{\varphi}(0)\left(\chi_{\gamma}-\chi_{\gamma'}\right)=\frac{\dot{\varphi}(0)}{\sigma^2(t_K)}\,\frac{1}{\alpha_K},
\end{equation}
where we assume that $\sigma(T_K)=\mathrm{area}B=1$.

\end{document}